\documentclass[reqno, 11pt]{amsart}

\makeatletter
\let\origsection=\section \def\section{\@ifstar{\origsection*}{\mysection}} 
\def\mysection{\@startsection{section}{1}\z@{.7\linespacing\@plus\linespacing}{.5\linespacing}{\normalfont\scshape\centering\S}}
\makeatother    

\usepackage{geometry}
\numberwithin{equation}{section}
\numberwithin{figure}{section}

\usepackage{xcolor}
\usepackage{hyperref}
\hypersetup{
    unicode,
    colorlinks,
    linkcolor={red!60!black},
    citecolor={green!60!black},
    urlcolor={blue!60!black}
}

\usepackage{amsfonts,amsthm,amssymb,amsmath}
\usepackage[T1]{fontenc}
\usepackage{microtype}
\usepackage{graphicx}
\usepackage{tikz}
\usepackage{tkz-berge}
\usepackage[usenames,dvipsnames]{pstricks}
\usepackage{epsfig}
\usepackage{verbatim}
\usepackage{scalerel}

\usepackage{enumerate,enumitem}
\usepackage{thmtools, thm-restate}%For duplicating theorem numbers.
\usetikzlibrary{decorations.pathmorphing, arrows, calc, matrix, arrows.meta}

\usepackage{mathtools}
\usepackage{subcaption}

\newtheorem{theorem}{Theorem}
\newtheorem{fact}{Fact}

\newcommand{\N}{\mathbb{N}}

\renewcommand{\triangleleft}{\vartriangleleft}
\renewcommand{\leq}{\leqslant}

\renewcommand{\rho}{\varrho}
\renewcommand{\subset}{\subseteq}
\renewcommand{\supset}{\supseteq}
\newcommand{\nottriangleleft}{\not\kern-1pt\mathrel{\triangleleft}}

\begin{document}
\title{Quickly proving Diestel's normal spanning tree criterion}   

\author{Max Pitz}
\address{Hamburg University, Department of Mathematics, Bundesstra\ss e 55 (Geomatikum), 20146 Hamburg, Germany}
\email{max.pitz@uni-hamburg.de}

%\keywords{normal spanning trees, minor, colouring number, excluded minor characterisation}

\begin{abstract}
We present two short proofs for Diestel's criterion that a connected graph has a normal spanning tree provided it contains no subdivision of a countable clique in which every edge has been replaced by uncountably many parallel edges.
\end{abstract}

\maketitle

\section{Overview}

This paper continues a line of inquiry started in \cite{pitz2020unified} with the aim to find efficient algorithms for constructing normal spanning trees in infinite graphs. A rooted spanning tree $T$ of a graph $G$ is called \emph{normal} if the end vertices of any edge of $G$ are comparable in the natural tree order of $T$. Intuitively, all the edges of $G$ run `parallel' to branches of $T$, but never `across'. %Since their introduction by Jung in 1969, normal spanning trees have developed to be perhaps the single most useful structural tool in infinite graph theory. 

Every countable connected graph has a normal spanning tree, but uncountable graphs might not, as demonstrated by complete graphs on uncountably many vertices. While exact characterisations of graphs with normal spanning trees exist, see e.g.\ \cite{jung1969wurzelbaume,pitz2020d}, these may be hard to verify in practice. The most applied sufficient condition  for normal spanning trees is the following criterion due to Halin~\cite{halin1978simplicial}, and its strengthening due to Diestel \cite{diestel2016simple}, see also \cite[\S6]{pitz2020d} for an updated proof.

\begin{theorem}[Halin, 1978]
\label{thm_halin}
Every connected graph without a $TK^{\aleph_0}$ has a normal spanning tree.
\end{theorem}

\begin{theorem}[Diestel, 2016]
\label{thm_diestel}
Every connected graph without fat $TK^{\aleph_0}$ has a normal spanning tree.
\end{theorem}

Here, a $TK^{\aleph_0}$ is any subdivision of the countable clique $K^{\aleph_0}$, and a \emph{fat $TK^{\aleph_0}$} is any subdivision of the multigraph obtained from a $K^{\aleph_0}$ by replacing every edge with $\aleph_1$ parallel edges.

Until recently, only fairly involved proofs of these results were available: Halin's original proof employing his theory of simplicial decompositions \cite{halin1978simplicial}, and Diestel's proof strategy building on the forbidden minor characterisation for normal spanning trees \cite{diestel2016simple,pitz2020d}. 

In  \cite{pitz2020unified}, however, the present author found a simple greedy algorithm which constructs the desired normal spanning tree in Halin's Theorem~\ref{thm_halin} in just $\omega$ many steps. The purpose of this  note is to provide two simple proofs also for Theorem~\ref{thm_diestel}, one of them again an $\omega$-length algorithm. % that works under the weaker assumptions of Theorem~\ref{thm_diestel}.

Notably, this algorithm also yields a new, local version of Theorem~\ref{thm_diestel}: Given a set of vertices $U$ of a connected graph $G$, there exists a normal tree of $G$ containing $U$ if and only if every fat $TK^{\aleph_0}$ in $G$ can be separated from $U$ by a finite set of vertices, see Theorem~\ref{cor_max2} below.

\section{Tree orders and normal trees}

We follow the notation in \cite{Bible}. The \emph{tree-order} $\leq_T$ of a tree $T$ with root $r$ is defined by setting $u \leq_T v$ if $u$ lies on the unique path from $r$ to $v$ in $T$. %Given a vertex $v$ of $T$, we denote by $T_v := T[\{t \colon v \leq_T t\}]$ the \emph{uptree of $T$ rooted in $v$}. 
For a vertex $v$ of $T$, let $\lceil v \rceil := \{t \in T \colon t \leq_T v\}$. % be the \emph{down-closure} $v$ in $T$.

For rooted trees that are not necessarily spanning, one generalises the notion of normality as follows: A rooted tree $T \subset G$ is \emph{normal (in $G$)} if the end vertices of any $T$-path in $G$ (a path in $G$ with end vertices in $T$ but all edges and inner vertices outside of $T$) are comparable in the tree order of $T$. If $T$ is spanning, this clearly reduces to the definition given in the introduction. If $T \subset G$ is normal, then the set of neighbours $N(D)$ of any component $D$ of $G - T$ forms a chain in $T$, i.e.\ all vertices of $N(D)$ are comparable in $\leq_T$. Moreover, incomparable nodes $v,w $ of any normal tree $T \subset G$ are separated in $G$ by $\lceil v \rceil \cap \lceil w \rceil$.

\begin{fact}[Jung {\cite[Satz~6]{jung1969wurzelbaume}}]
\label{fact_Jung}
Let $G$ be a graph with a normal spanning tree. Then for every connected subgraph $C \subset G$ and every $r \in C$ there is a normal spanning tree of $C$ with root $r$.
\end{fact}

For distinct vertices $v,w$ of $G$ we denote by $\kappa(v,w) = \kappa_G(v,w)$ the connectivity between $v$ and $w$ in $G$, i.e.\  the largest size of a family of independent $v-w$ paths. If $v$ and $w$ are non-adjacent, this is by Menger's theorem equivalent to the minimal size of a $v-w$ separator in $G$.

\begin{fact}[Halin, {\cite[(15)]{halin1967unterteilungen}}]
\label{fact_Halin}
Let $U$ be a countable set of vertices in $G$. There is a fat $TK^{\aleph_0}$ with branch vertices $U$ if and only if $\kappa(u,v)$ is uncountable for all $u \neq v \in U$.
\end{fact}

%Lastly, we denote ordinals by $i,j,\ell,\sigma$.

\section{The first proof}

%Our first proof uses the language of elementary submodels.

\begin{proof}[First proof of Theorem~\ref{thm_diestel}]
By induction on $|G|$. We may assume that $|G|$ is uncountable. Suppose we have a continuous increasing ordinal-indexed sequence $(G_i \colon i < \sigma)$ of induced subgraphs all of size less than $|G|$ with $G = \bigcup_{i < \sigma} G_i$ such that
%We use the following well-known facts:
\begin{enumerate}[label=$(\roman*)$]
\item \label{item_1} the end vertices of any $G_i$-path in $G$ have infinite connectivity in $G_i$, and
\item \label{item_2} the end vertices of any $G_i$-path in $G$ have uncountable connectivity in $G$. 
\end{enumerate}

Then we can construct normal spanning trees $T_i$ of $G_i$ extending each other all with the same root by (transfinite) recursion on $i$. If $\ell < \sigma$ is a limit, we may simply define $T_\ell = \bigcup_{i < \ell} T_i$. For the successor case, suppose that $T_i$ is already defined. By \ref{item_2}, the neighbourhood $N(C)$ is finite for every component $C$ of $G_{i+1} - G_i$ (otherwise we get a fat $TK^{\aleph_0}$ by Fact~\ref{fact_Halin}), and by \ref{item_1}, $N(C)$ lies on a chain of $T_i$ (as incomparable vertices in $T_i$ are separated in $G_i$ by the intersection of their finite down-closures). %Since $|G_{i+1}| < |G|$, b
Let $t_C \in N(C)$ be maximal in the tree order of $T_i$, and let $r_C$ be a neighbour of $t_C$ in $C$. By the induction hypothesis and Fact~\ref{fact_Jung}, $C$ has a normal spanning tree $T_C$ with root $r_C$. Then $T_i$ together with all $T_C$ and edges $t_Cr_C$ is a normal spanning tree $T_{i+1}$ of $G_{i+1}$.
Once the recursion is complete, $T = \bigcup_{i < \sigma} T_i$ is the desired normal spanning tree of $G$.

It remains to construct a sequence $(G_i \colon i < \sigma)$ with \ref{item_1} and \ref{item_2}. This can be done, for example, by taking a continuous increasing chain $(M_i \colon i < \sigma)$ with $\sigma = cf(|G|)$ of ${<}|G|$-sized elementary submodels $M_i$ of %$H(\lambda^+)$ 
a large enough fragment of ZFC with $G \in M_{i}$, such that $G \subset  \bigcup_{i < \sigma} M_i$, see \cite{soukup2011elementary}. Then $G_i = G \cap M_i$ is as required. %\qedhere
Alternatively, %for a hands-on proof, 
use a countable closure argument to construct $G_i$ such that for every pair $v,w \in V(G_i)$ with $\kappa_G(v,w) \leq \aleph_0$, the graph $G_i$ contains a maximal family of independent $v-w$ paths in $G$ (this will guarantee \ref{item_2}), and for all other pairs, $G_i$ contains at least countably many independent  $v-w$ paths (this will guarantee \ref{item_1}); and note that properties \ref{item_1} and \ref{item_2} are preserved under increasing unions.
\end{proof}

\section{The second proof}
\label{sec_proof2}

Our second proof extracts the closure properties \ref{item_1} and \ref{item_2} in the previous construction, and combines them into a single algorithm constructing the normal spanning tree in $\omega$ many steps, avoiding  ordinals and transfinite constructions altogether.

\begin{proof}[Second proof of Theorem~\ref{thm_diestel}]
%For the forwards implication, recall that the levels of any normal spanning tree are dispersed, and hence in particular fat $TK^{\aleph_0}$-dispersed in $G$.

%Conversely, let $G$ be a connected graph and let $\{V_n \colon n \in \N\}$ be a collection of fat $TK^{\aleph_0}$-dispersed sets in $G$ with $V(G) = \bigcup_{n \in \N} V_n$. 
For every pair of distinct vertices $v$ and $w$ of $G$ with $\kappa(v,w)$ at most countable, fix a maximal collection $\mathcal{P}_{v,w} = \{P_{v,w}^1,P_{v,w}^2,\ldots\}$ of independent $v{-}w$ paths in $G$.

Construct a countable chain $T_0 \subset T_1 \subset T_2 \subset \cdots$ of rayless normal trees in $G$ with the same root $r \in V(G)$ as follows: Put $T_0 = \{r\}$, and suppose $T_n$ has already been defined. Since $T_n$ is a rayless normal tree, any component $D$ of $G - T_n$ has a finite neighbourhood $N(D)$ in $T$. For each pair $v \neq w \in N(D)$ with countable connectivity select the path $P^D_{v,w}$ with least index $\mathcal{P}_{v,w}$ intersecting $D$. By \cite[Proposition~1.5.6]{Bible}, we may extend $T_n$ finitely into every such component $D$ as to cover $P^D_{v,w} \cap D$ for all $v \neq w \in N(D)$ (or at least one abitrarily chosen vertex making the extension into $D$ is non-trivial), so that the extension $T_{n+1} \supset T_n$ is a rayless normal tree with root $r$. This completes the construction. 

The union $T = \bigcup_{n \in \N} T_n$ with root $r$ is a normal tree in $G$. We claim that $T$ is spanning unless  $G$ contains a fat $TK^{\aleph_0}$.
If $T$ is not spanning, consider a component $C$ of $G-T$. %, and let $n_C$ be minimal such that $V_{n_C} \cap C \neq \emptyset$. 
Then $N(C) \subseteq T$ is infinite: otherwise, $N(C) \subseteq T_n$ for some $n \in \N$ but then we would have extended $T_n$ into $C$, a contradiction. %Hence, $N(C)$ lies on a unique ray $R \subset T$ starting at the root of $T$. 
For every $n$, let $D_n$ be the unique component of $G - T_n$ containing $C$.

By Fact~\ref{fact_Halin}, it suffices to show that $\kappa(v,w)$ is uncountable for every $v \neq w \in N(C)$. 
 Consider a $T$-path $P$ from $v$ to $w$ with $\mathring{P} \subset C$. If $\kappa(v,w)$ was countable, then by maximality of $\mathcal{P}_{v,w}$ there is $P_{v,w}^k \in \mathcal{P}_{v,w}$ with say $P^k_{v,w} \cap \mathring{P} \ni x$. Let $m$ be minimal with $v,w \in T_{m}$. Since the $P^{D_n}_{v,w}$ are pairwise distinct, the path $P^k_{v,w}$ was selected as $P^{D_n}_{v,w}$ for some $n$ with $m  \leq n \leq m+k$. But then 
$x \in P^k_{v,w} \cap \mathring{P} \subset P^{D_n}_{v,w} \cap D_n \subset T_{n+1} \subset T$
%contradicting that $P$ is a $T$-path. 
contradicts that $P$ is a $T$-path.
\end{proof}

\section{Local versions of Diestel's criterion}

By a slight modification of this $\omega$-length algorithm, %with the methods from \cite{pitz2020unified}, 
one readily obtains a proof of the following results, which answer \cite[Problem~3]{pitz2020d}.

\begin{theorem}
\label{cor_max2}
A set of vertices $U$ in a connected graph $G$ is contained in a normal tree of $G$ provided %$U$ is fat $TK^{\aleph_0}$-dispersed.
 every fat $TK^{\aleph_0}$ in $G$ can be separated from $U$ by a finite set of vertices.
\end{theorem}

\begin{proof}
Let $U$ be a set of vertices such that every fat $TK^{\aleph_0}$ in $G$ can be separated from $U$ by a finite set of vertices. Use the algorithm from Section~\ref{sec_proof2}, but only extend $T_n$ into a component $D$ of $G - T_n$ with $U \cap D \neq \emptyset$. Additionally, make sure to cover at least one vertex from $U \cap D$. 

It remains to argue that $U$ is contained in $T = \bigcup T_n$. Otherwise, there is a component $C$ of $G-T$ containing a vertex from $U$. As in Section~\ref{sec_proof2}, this gives us a fat $TK^{\aleph_0}$ in $G$ which furthermore cannot be separated from $U$ by a finite set of vertices, cf.\  \cite{pitz2020unified}.
\end{proof}

\begin{theorem}
\label{cor_max1}
A connected graph has a normal spanning tree if and only if its vertex set is a countable union of sets each separated from any fat $TK^{\aleph_0}$ by a finite set of vertices.
\end{theorem}

\begin{proof}
For the forward implication, recall that the levels of any normal spanning tree can be separated by a finite set of vertices from any ray, and hence in particular from any fat $TK^{\aleph_0}$. Conversely, let $\{V_n \colon n \in \N\}$ be a collection of fat $TK^{\aleph_0}$-dispersed sets in $G$ with $V(G) = \bigcup_{n \in \N} V_n$. Adapt the algorithm from Section~\ref{sec_proof2}, so that when extending $T_n$ into a component $D$ of $G - T_n$, we additionally cover a vertex $v_D \in D \cap V_{n_D}$ where $n_D$ minimal such that $V_{n_D} \cap D \neq \emptyset$. The proof then proceeds as in \cite{pitz2020unified}.
\end{proof}

\bibliographystyle{plain}
\bibliography{reference}

\begin{thebibliography}{1}

\bibitem{Bible}
Reinhard Diestel.
\newblock {\em {Graph Theory}}.
\newblock Springer, 5th edition, 2015.

\bibitem{diestel2016simple}
Reinhard Diestel.
\newblock A simple existence criterion for normal spanning trees.
\newblock {\em The Electronic Journal of Combinatorics}, 2016.
\newblock P2:33.

\bibitem{halin1967unterteilungen}
Rudolf Halin.
\newblock Unterteilungen vollst{\"a}ndiger {G}raphen in {G}raphen mit
  unendlicher chromatischer {Z}ahl.
\newblock In {\em Abhandlungen aus dem Mathematischen Seminar der
  Universit{\"a}t Hamburg}, volume~31, pages 156--165. Springer, 1967.

\bibitem{halin1978simplicial}
Rudolf Halin.
\newblock Simplicial decompositions of infinite graphs.
\newblock In {\em Annals of Discrete Mathematics}, volume~3, pages 93--109.
  Elsevier, 1978.

\bibitem{jung1969wurzelbaume}
Heinz~A. Jung.
\newblock Wurzelb{\"a}ume und unendliche {W}ege in {G}raphen.
\newblock {\em Mathematische Nachrichten}, 41(1-3):1--22, 1969.

\bibitem{pitz2020d}
Max Pitz.
\newblock Proof of {H}alin's normal spanning tree conjecture.
\newblock https://arxiv.org/abs/2005.02833, 2020.

\bibitem{pitz2020unified}
Max Pitz.
\newblock A unified existence theorem for normal spanning trees.
\newblock https://arxiv.org/abs/2003.11575, 2020.

\bibitem{soukup2011elementary}
Lajos Soukup.
\newblock Elementary submodels in infinite combinatorics.
\newblock {\em Discrete mathematics}, 311(15):1585--1598, 2011.

\end{thebibliography}

\end{document}